\documentclass{amsart}

\usepackage{graphicx}

\newtheorem{theorem}{Theorem}[section]

\newtheorem{proposition}[theorem]{Proposition}

\theoremstyle{definition}
\newtheorem{definition}[theorem]{Definition}

\theoremstyle{remark}
\newtheorem{remark}[theorem]{Remark}

\numberwithin{equation}{section}

\def\Plambda{(2.13)_{\lambda}}
\def\Pext{(2.13)_{\lambda^\star}}

\begin{document}

\title[Geometric-type Sobolev inequalities and 
applications]{Geometric-type Sobolev inequalities and 
applications to the regularity of minimizers}

\author{Xavier Cabr\'e}
\address{ICREA and Universitat Polit\`ecnica de Catalunya, 
Departament de Matem\`{a}tica  Aplicada I, Diagonal 647, 08028 Barcelona, Spain}
\email{xavier.cabre@upc.edu}
\thanks{The authors 
were supported by grants MTM2008-06349-C03-01, MTM2011-27739-C04-01 (Spain) and 
2009SGR345 (Catalunya)}

\author{Manel Sanch\'on}
\address{Universitat de Barcelona, Departament de Matem\`atica Aplicada i An\`alisi,
Gran Via 585, 08007 Barcelona, Spain}
\email{msanchon@maia.ub.es}


\date{}



\begin{abstract}
The purpose of this paper is twofold. We first prove a
weighted Sobolev inequality and part of a weighted Morrey's 
inequality, where the weights are a power of the mean curvature 
of the level sets of the function appearing in the inequalities. 
Then, as main application of our inequalities, we establish new 
$L^q$ and $W^{1,q}$ estimates for semi-stable solutions of 
$-\Delta u=g(u)$ in a bounded domain $\Omega$ of $\mathbb{R}^n$.  
These estimates lead to an $L^{2n/(n-4)}(\Omega)$ bound for the 
extremal solution of $-\Delta u=\lambda f(u)$ when $n\geq 5$ 
and the domain is convex. We recall that extremal solutions are 
known to be bounded in convex domains if $n\leq 4$, and that
their boundedness is expected ---but still unkwown--- 
for $n\leq 9$.
\end{abstract}

\maketitle

\section{Introduction}\label{section1}
The main purpose of this paper is twofold. On the one hand, we prove 
the following geometric-type Sobolev and Morrey's inequalities for functions 
$v\in C_0^\infty(\overline{\Omega})$, where $\Omega$ is a smooth bounded 
domain of $\mathbb{R}^n$ with $n\geq 2$.
Assume that $p\geq 1$ and $r\in\{0\}\cup[1,\infty)$. 
Then, there exists a constant $C$ depending 
only on $n$, $p$, and $r$, such that the following inequalities hold
for all $v\in C_0^\infty(\overline{\Omega})$:
$$
\|v\|_{L^{p^\star_r}(\Omega)}\leq C\Big\||H_v|^r|\nabla v|\Big\|_{L^p(\Omega\cap\{|\nabla v|>0\})}
\qquad\textrm{if }n>p(1+r)
$$
and
$$
\|v\|_{L^{\infty}(\Omega)}\leq C|\Omega|^\frac{p(1+r)-n}{np}
\Big\||H_v|^r|\nabla v|\Big\|_{L^p(\Omega\cap\{|\nabla v|>0\})}
\qquad\textrm{if }1+r\leq n<p(1+r).
$$
Here, the critical exponent $p_r^\star$ is defined by 
$$
\frac{1}{p^\star_r}:=\frac{1}{p}-\frac{1+r}{n}
$$ 
and the function $H_v$ appearing in the right hand side of both inequalities
denotes the mean curvature of the level sets of $|v|$ (which are 
smooth hypersurfaces at points where $|\nabla v|>0$). In particular, 
it depends on $v$ in a nonlinear way, given by the expression
$$
H_v=\frac{-1}{n-1}\ {\rm div}\left(\frac{\nabla v}{|\nabla v|}\right).
$$
We also establish a related inequality of Trudinger type when $n=p(1+r)$.

On the other hand, as an application of these inequalities, we derive 
new  $L^q$ and $W^{1,q}$ \textit{a priori} estimates for local minimizers 
(and more generally, for semi-stable solutions) of reaction-diffusion problems. 
These estimates motivated the study of the geometric Sobolev inequalities above.
 
Consider the reaction-diffusion problem
\begin{equation}\label{problem}
\left\{
\begin{array}{rcll}
-\Delta u&=&g(u)&\textrm{in }\Omega,\\
u&=&0&\textrm{on }\partial \Omega,
\end{array}
\right.
\end{equation}
where $g$ is any $C^1$ function. We say that a classical solution 
$u\in C^2(\overline{\Omega})$ of \eqref{problem} is \textit{semi-stable} 
if 
\begin{equation}\label{semi-stab1}
\int_\Omega \{|\nabla\xi|^2-g'(u)\xi^2\}\ dx\geq 0
\qquad\textrm{for all }\xi\in C_0^1(\overline{\Omega}).
\end{equation}
This class of solutions includes local minimizers of the associated 
energy functional, minimal solutions, extremal solutions, and also 
certain solutions found between a sub and a super solution. We use 
the semi-stability condition \eqref{semi-stab1} with the test function 
$\xi=|\nabla u|\eta$. Using this choice of $\xi$ and an equation for
$(\Delta+g'(u))|\nabla u|$, one deduces that
\begin{equation}\label{semi1}
(n-1)\int_{\Omega\cap\{|\nabla u|>0\}} H_u^2|\nabla u|^2\eta^2\ dx
\leq \int_\Omega |\nabla u|^2|\nabla\eta|^2 \ dx
\end{equation}
for every Lipschitz function $\eta$ in $\overline\Omega$ with 
$\eta|_{\partial\Omega}\equiv0$. We take $\eta\equiv 1$ in a 
compact set $K\subset\Omega$, and thus $|\nabla\eta|$ is supported 
in $\overline{\Omega}\setminus K$. Then, if we know that $u$ is regular in a 
neighborhood of $\partial\Omega$ (this holds for instance when $\Omega$ 
is convex) and we take $K$ big enough, the right hand side of 
\eqref{semi1} is bounded. We deduce that 
$$
\int_{K\cap\{|\nabla u|>0\}} H_u^2|\nabla u|^2\ dx\leq C
$$
and, with the help of our Sobolev inequality above with $r=1$ and $p=2$, we 
establish a new bound:
$$
u\in L^{2n/(n-4)}(\Omega)\quad \textrm{ if }n\geq 5\textrm{ and }\Omega\textrm{ is convex}.
$$
Moreover, using this $L^{2n/(n-4)}$ estimate, we are also able to 
obtain $W^{1,q}$ bounds for semi-stable solutions.
This result completes the $L^\infty$ estimate obtained by the 
first author in \cite{Cabre09} whenever $n\leq 4$ and $\Omega$ is convex. 

For general domains and increasing positive and convex nonlinearities $g$, Nedev \cite{Nedev} 
proved an $L^\infty$ bound when $n\leq 3$, and an $L^q$ estimate, for every
$q<n/(n-4)$ when $n\geq 4$. Note that the exponent $2n/(n-4)$ in our $L^q$ bound above 
improves the one of Nedev. Besides, we make no assumption on the nonlinearity,
but in contrast with Nedev's result, we assume $\Omega$ to be convex.

\section{Main results}

\subsection{Geometric-type Sobolev inequalities}\label{subsection1:1}

We start stating the Sobolev and Morrey's type inequalities 
involving the mean curvature of the level sets.
\begin{theorem}\label{Theorem:Sobolev}
Let $\Omega$ be a smooth bounded domain of $\mathbb{R}^n$, with $n\geq 2$. 
Let $p\geq 1$ and $r\in\{0\}\cup[1,\infty)$.

Let $v\in C^\infty_0(\overline{\Omega})$ $($\textit{i.e.}, 
$v\in C^\infty(\overline{\Omega})$ and $v=0$ on $\partial\Omega)$. 
For $x\in\Omega$ with $\nabla v(x)\neq 0$, let $H_v(x)$ be the mean curvature 
at $x$ of the hypersurface $\{y\in\Omega: |v(y)|=|v(x)|\}$, which is smooth at $x$. 
The following assertions hold:

\begin{enumerate}
\item[$(a)$] Assume either that $1+r\leq n<p(1+r)$ or that $n=1+r$ and $p=1$. Then 
\begin{equation}\label{Sobolev:infinity}
\|v\|_{L^\infty(\Omega)}
\leq C_1|\Omega|^{\frac{p(1+r)-n}{np}}
\left(\int_{\Omega\cap\{|\nabla v|>0\}} |H_v|^{pr}|\nabla v|^p \ dx\right)^{1/p},
\end{equation}
for some constant $C_1$ depending only on $n$, $p$, and $r$.

\item[$(b)$] If $n>p(1+r)$, then 
\begin{equation}\label{Sobolev}
\left(\int_{\Omega}|v|^{p_r^\star}\ dx\right)^{1/p_r^\star}
\leq C_2\left(\int_{\Omega\cap\{|\nabla v|>0\}} |H_v|^{pr}|\nabla v|^p \ dx\right)^{1/p},
\end{equation}
where $\displaystyle \frac{1}{p^\star_r}:=\frac{1}{p}-\frac{1+r}{n}$, 
for some constant $C_2$ depending only on $n$, $p$, and~$r$.

\item[$(c)$] If $p>1$ and $n=p(1+r)$, then
\begin{equation}\label{quasilinear:b2}
\int_\Omega\exp\left\{\left(\frac{|v|}{C_3 (\int_{\Omega\cap\{|\nabla v|>0\}}
|H_v|^{pr}|\nabla v|^p\ dx)^{1/p}}\right)^{p'}\right\}\ dx
\leq C_4|\Omega|,
\end{equation}
where $p'=p/(p-1)$, and $C_3$ and $C_4$ are positive constants
depending only on $n$ and $p$.
\end{enumerate}
\end{theorem}

In Remarks \ref{Remark:Sobolev} and \ref{Remark:Sobolev2} we give 
explicit expressions for admissible values of the constants $C_i$, 
$i=1,...,4$, in the theorem. These expressions involve two 
isoperimetric constants $A_1$ and $A_2$ (only $A_1$ when $r=0$) that 
we describe next. 

Note that Theorem~\ref{Theorem:Sobolev} is well known for $r=0$. 
Indeed, $(a)$ states a part of Morrey's inequality, 
$(b)$ is the classical Sobolev inequality, and $(c)$ 
is Trudinger's inequality. It is well known that 
they follow from the classical isoperimetric inequality, 
which states that for any smooth bounded domain $D$ of $\mathbb{R}^n$,
\begin{equation}\label{isopctant}
A_1|D|^{(n-1)/n}\leq |\partial D|
\end{equation}
where $A_1=n|B_1|^{1/n}$ and $B_1$ denotes the unit ball in $\mathbb{R}^n$. 
Our proof will show this fact and that 
admissible constants in the theorem are completely explicit in 
terms only of $A_1$, $n$, and $p$ when $r=0$. 

To establish the theorem when $r\geq1$ we need another 
isoperimetric inequality. It involves the mean curvature $H$ of $C^2$ immersed 
$(n-1)$-dimensional compact hypersurfaces without boundary $S\subset 
\mathbb{R}^{n}$, and states
\begin{equation}\label{isop:mean}
|S|^{\frac{n-2}{n-1}} \leq A_2\int_S |H(x)|\ d\sigma.
\end{equation}
Here, $H$ is the mean curvature of $S$, $d\sigma$ denotes the area 
element in $S$, and $A_2$ is a universal constant depending only on the dimension 
$n\geq 2$. When $n=2$, \eqref{isop:mean} follows from the Gauss-Bonnet 
formula. When $n\geq 3$, the inequality is due to Michael and Simon~\cite{MS} 
and to Allard \cite{A} ---see Theorem~28.4.1 \cite{Burago} for a more general 
version of \eqref{isop:mean}. From such a version, a Sobolev 
inequality for functions defined on hypersurfaces $S$ of $\mathbb{R}^n$, and which 
involves the mean curvature $H$ of $S$, can be deduced (see section 28.5 of \cite{Burago}, 
Theorem~2.1~\cite{Cabre09}, or Theorem~C.2.1~\cite{Dupaigne}).

\begin{remark}[{\it The critical exponents}\rm]\label{critical}
Note that the critical exponent $p_r^\star$ in part $(b)$ of the theorem coincides 
with the classical Sobolev exponent in the embedding $W^{1+r,p}\subset
L^{p_r^\star}$ for functions with $1+r$ derivatives in $L^p$.

The critical case in Theorem~\ref{Theorem:Sobolev} corresponds to 
$n=p(1+r)$. It is given in part $(a)$ when $p=1$ and 
in part $(c)$ when $p>1$. In the second case, $p>1$, the $L^\infty$ estimate 
does not necessarily hold, as usual. This can be easily seen using radial 
functions when $\Omega$ is a ball. Instead, the embedding in 
$L^\infty$ holds in the critical case when $p=1$ (and thus $n=1+r$), 
as in the classical case $W^{n,1}\subset L^\infty$.

Note that in all cases of Theorem~\ref{Theorem:Sobolev} we have $1+r\leq n$. 
In the case $p=1$ and $n<1+r$, which is not covered by Theorem~\ref{Theorem:Sobolev}, 
we derive an inequality involving the total variation of $|v|$ in Remark~\ref{Rmk:perimeter}. 
\end{remark}

\begin{remark}[{\it The case $p=+\infty$}\rm]\label{r:p_inf}
Letting $p$ tend to $+\infty$ in \eqref{Sobolev:infinity} and using the explicit 
constant $C_1$ obtained in Remark \ref{Remark:Sobolev}, we deduce
\begin{equation}\label{sharpext}
\|v\|_{L^\infty(\Omega)}
\leq 
\frac{n}{1+r}\Big(n|B_1|^{1/n}\Big)^{\frac{1+r-n}{n-1}}A_2^r|\Omega|^\frac{1+r}{n}
\ \||H_v|^r|\nabla v|\|_{L^\infty(\Omega\cap\{|\nabla v|>0\})}
\end{equation}
when $n\geq 2$ and $1\leq r \leq n-1$. Here, $A_2$ is a constant 
depending only on the dimension $n$ for which \eqref{isop:mean} holds. 
\end{remark}

\begin{remark}[{\it The case $r\in(0,1)$}\rm]\label{r01}
Theorem~\ref{Theorem:Sobolev} is stated for $r=0$ and $r\geq 1$. 
A natural question is if it does not hold for $r\in(0,1)$ independently 
of the dimension~$n$. In this direction, in 
Remark~\ref{stadium} we prove that Theorem~\ref{Theorem:Sobolev} $(a)$-$(b)$ do not 
hold for $r\in(0,2p^{-1}-1)$ when $1\leq p <2$, independently of the dimension. 
In particular, they do not hold for $r\in(0,1)$ when $p=1$. 
\end{remark}

For the class of mean convex functions
---that is, functions whose level sets 
have nonnegative mean curvature--- the estimates in Theorem~\ref{Theorem:Sobolev} 
can be established in the larger range $r\geq 1/p$.
The argument only applies to mean convex 
functions since it relies on the fact that the perimeter of the 
level sets of a mean convex function~$v$ is a nonincreasing function, 
\textit{i.e.}, $|\{x\in\Omega: |v(x)|=t_1\}|\geq 
|\{x\in\Omega: |v(x)|=t_2\}|$ for a.e. $0<t_1<t_2$. 
When $r=1/p$, such estimates were proven by Trudinger~\cite{Trudinger97}.
The inequalities in \cite{Trudinger97} carry optimal constants  
and are claimed there to hold for all mean convex
functions. However, at present they are only known to hold for functions
with {\it starshaped} and mean convex level sets. The reason is that
to obtain optimal constants one needs to use inequality \eqref{isop:mean}
with the constant $A_2$ which makes \eqref{isop:mean} to be 
an equality when $S$ is a sphere. That such constant $A_2$ is admissible 
in \eqref{isop:mean} is still only known among {\it starshaped} mean convex
hypersurfaces~$S$, by a recent result of Guan and Li~\cite{GL}; see also~\cite{GMTX}.

Theorem~\ref{Theorem:Sobolev} can be used to study the geometric flow
of mean convex hypersurfaces driven by a positive power $r$ of their mean curvature,
the so-called $H^r$-flow.
The theorem leads, for instance, to upper bounds on the extinction time
of the flow. In the level set formulation, the flow can be represented by the
level sets of a mean convex function $v$ satisfying the elliptic equation
$$
H_v=\frac{-1}{n-1}\ {\rm div}\left(\frac{\nabla v}{|\nabla v|}\right)=\frac{1}{|\nabla v|^{1/r}}.
$$
Noting that $\||H_v|^r|\nabla v|\|_{L^\infty(\Omega\cap\{|\nabla v|>0\})}=1$ 
and using \eqref{sharpext} one obtains an $L^\infty$ estimate for $v$, 
or equivalently, an upper bound for the extinction time of the $H^r$-flow.
Let us mention here that Schulze~\cite{S08} used the ${H}^r$-flow
to give a new proof of a deep result
of B. Kleiner: the Euclidean isoperimetric inequality also 
holds for domains of any complete and simply-connected 3-dimensional manifold
with nonpositive sectional curvatures ---a result that is still open for 
the same type of manifolds of dimension $n\geq 5$.

In this respect, Theorem~\ref{Theorem:Sobolev} could be extended to the case of functions defined on 
Riemannian manifolds. Indeed, the first ingredient in our proof ---the coarea formula--- holds 
on any Riemannian manifold. On the other hand, the isoperimetric inequalities that 
we use to prove the theorem could be replaced by those in the particular
manifold; see section 36.5 of \cite{Burago}.

\begin{remark}[{\it The radial case}\rm]\label{rmk:radial}
When $\Omega=B_R=B_R(0)$, if we restrict inequality \eqref{Sobolev} 
to radially symmetric functions $v$ with compact support in $B_R$ 
then \eqref{Sobolev} reads 
\begin{equation}\label{radial}
\left(\int_0^R |v(\rho)|^q \rho^{n-1}\ d\rho\right)^{1/q}\leq 
C\left(\int_0^R \rho^{-pr}|v'(\rho)|^p\ \rho^{n-1}\ d\rho\right)^{1/p},
\end{equation}
where $q=p_r^\star$. Here $\rho=|x|$. Note that in the radial case, the level set 
at $x$, $\{|v|=|v(x)|\}$, is a sphere of radius $|x|$, and thus 
the average of its principal curvatures is $H_v(x)=|x|^{-1}=\rho^{-1}$.
The 1-dimensional weighted Sobolev inequality \eqref{radial} 
has been well studied (see \cite{Kufner} for this one and more general 
versions). It is well known that \eqref{radial} holds, with a constant $C$
independent of $v$, if and only if either $n< p(1+r)$ and $q\leq +\infty$, 
or $n>p(1+r)$ and $q\leq p_r^\star$, 
or $n= p(1+r)$ and $q<+\infty$.
This shows that Theorem~\ref{Theorem:Sobolev}~$(b)$ is sharp in terms 
of the exponents that it involves and the restrictions on them. The 
sharpness in this same sense of parts $(a)$ and $(c)$ of 
Theorem~\ref{Theorem:Sobolev} can also be checked using radially  
decreasing functions.
\end{remark}

\begin{remark}[{\it Relation with a Caffarelli-Kohn-Nirenberg inequality}\rm]\label{CaKoNi}
Since $H_v(x)=|x|^{-1}$ for radial functions, Theorem~\ref{Theorem:Sobolev} $(b)$ is 
related to the Caffarelli, Kohn, and Nirenberg inequality \cite{CKN}, which 
states the following. 
\textit{Assume $q>0$, $p\geq 1$, and $n>pr$. 
Then, there exists a positive constant $C$ such that 
\begin{equation}\label{ineq:CKN}
\|v\|_{L^q(\mathbb{R}^n)}\leq C\||x|^{-r}|\nabla v|\|_{L^p(\mathbb{R}^n)}
\end{equation}
holds for all $v\in C_0^\infty(\mathbb{R}^n)$,
if and only if $q=p_r^\star$ and $-1\leq r\leq 0$.}
Here, the condition $r\leq 0$ is due to the unboundedness 
of the domain and to the fact that the singularity of the weight is fixed 
at the origin ---and thus \eqref{ineq:CKN} is not invariant under translations. 
Indeed, that $r\leq 0$ is necessary in \eqref{ineq:CKN} can be shown by 
taking $v(x)=u(x-x_0)$ with $u\in C_0^\infty(B_1)$ and letting $|x_0|\rightarrow 
+\infty$.

Instead, our inequalities are invariant under translations.
\end{remark}

The second part of this paper is devoted to obtain, as an application of 
Theorem~\ref{Theorem:Sobolev}, \textit{a priori} estimates for semi-stable 
solutions of the reaction-diffusion problem \eqref{problem} ---which motivated 
the present work.

\subsection{Application to the regularity of stable solutions and extremal 
solutions}\label{subsection1:2}
Applying Theorem~\ref{Theorem:Sobolev} we obtain \textit{a priori} estimates 
for semi-stable solutions of \eqref{problem}. In particular, for the extremal 
solution $u^\star$ of $\Plambda$ below ---\textit{i.e.}, problem \eqref{problem} 
when $g(u)=\lambda f(u)$.

Recently, the first author proved the boundedness of 
the extremal solution of $\Plambda$ when the domain is convex and $n\leq 4$. 
Our following result is the main application of Theorem~\ref{Theorem:Sobolev}.
We establish an $L^\frac{2n}{n-4}$ estimate for the extremal 
solution in convex domains when $n\geq 5$. For these domains, the result improves 
the $L^q$ for $q<n/(n-4)$ and the $L^\frac{2n}{n-2}$ estimates of Nedev  proved, 
respectively, in \cite{Nedev} and \cite{Nedev01}.

\begin{theorem}\label{Theorem2}
Let $f:[0,+\infty)\longrightarrow\mathbb{R}$ be an increasing positive $C^1$ function 
$($in particular with $f(0)>0)$ such that 
$f(t)/t\rightarrow+\infty$ as $t\rightarrow+\infty$. Assume 
that $\Omega$ is a convex smooth bounded domain of $\mathbb{R}^n$ with 
$n\geq 5$. Let $u^\star$ be the extremal solution of $\Plambda$. 
Then, $$u^\star\in L^\frac{2n}{n-4}(\Omega).$$
\end{theorem}

The convexity assumption on the domain $\Omega$ is only used to control the 
$L^\infty$-norm of $u^\star$ in a neighborhood of the boundary $\partial\Omega$. 
For general domains and general nonlinearities we are able to prove the 
following \textit{a priori} estimates for semi-stable solutions
---from which Theorem~\ref{Theorem2} will follow easily.
\begin{theorem}\label{Theorem}
Let $g$ be any $C^\infty$ function and $\Omega\subset\mathbb{R}^n$ any smooth 
bounded domain with $n\geq5$. Let $u\in C^1_0(\overline{\Omega})$ be a semi-stable 
solution of \eqref{problem}, \textit{i.e.}, a solution satisfying \eqref{semi-stab1}. 
Then, 
\begin{equation}\label{Lq:estimate}
\left(\int_{\{|u|>s\}} \Big(|u|-s\Big)^{\frac{2n}{n-4}}\ dx\right)^\frac{n-4}{2n} 
\leq \frac{C(n)}{s}\left(\int_{\{|u|\leq s\}} |\nabla u|^4 \ dx\right)^{1/2}
\end{equation}
for all $s>0$, where $C(n)$ is a constant depending only on $n$. Moreover, 
\begin{equation}\label{grad:estimate}
\int_\Omega |\nabla u|^p\ dx\leq p|\Omega|+ \left(\frac{4n}{(3n-4)p}-1\right)^{-1}
\left\{
\int_\Omega|u|^{\frac{2n}{n-4}}\ dx+\|g(u)\|_{L^1(\Omega)}
\right\}
\end{equation}
for all $1\leq p<\frac{4n}{3n-4}$.
\end{theorem}

Inequality \eqref{Lq:estimate} is relevant since the set $\{|u|\leq s\}$ on its 
right hand side is a small neighborhood of $\partial \Omega$ (at least if $u>0$ 
in $\Omega$) if $s$ is chosen small enough. Thus the $L^{2n/(n-4)}(\Omega)$ bound 
gets reduced to a question on the regularity of $u$ near $\partial\Omega$.

To prove Theorem~\ref{Theorem} we take the truncation 
of $|u|$ at level $s$ as a test function in \eqref{semi1}
to obtain
\begin{equation}\label{eq1:12}
(n-1)s^2\int_{\{|u|>s\}\cap\{|\nabla u|>0\}}H_u^2|\nabla u|^2\ dx
\leq
\int_{\{|u|\leq s\}}|\nabla u|^4\ dx.
\end{equation}
Now, \eqref{Lq:estimate} follows from \eqref{eq1:12} and 
our geometric Sobolev inequality \eqref{Sobolev} 
with $p=2$ and $r=1$. 

When $2 \leq n\leq 3$,  from \eqref{eq1:12} and Theorem~\ref{Theorem:Sobolev} $(a)$, 
it follows that
\begin{equation}\label{eq1:13}
\|u\|_{L^\infty(\Omega)}\leq s+\frac{C(n)}{s}
|\Omega|^\frac{4-n}{2n}\left(\int_{\{|u|\leq s\}}|\nabla u|^4\ dx\right)^{1/2},
\end{equation}
where $C(n)$ is a constant depending only on $n$. The \textit{a priori}
estimate \eqref{eq1:13} was proved by the first author in \cite{Cabre09} in a different 
way, obtaining the $L^\infty$ estimate also in dimension~$4$. 

The gradient estimate \eqref{grad:estimate} follows 
from the $L^{2n/(n-4)}$ bound  
with the aid of a technique introduced by B\'enilan \textit{et al.} \cite{BBGGPV95} 
to prove regularity of entropy solutions for $p$-Laplace equations with 
$L^1$ data (see Proposition~\ref{Thm:bootstrap} below).

The study of the regularity of semi-stable solutions was motivated by 
some open problems raised by Brezis and V\'azquez \cite{BV} about the regularity 
of extremal solutions. 
They appear in the following context. Consider positive solutions of 
\stepcounter{equation}
$$
\left\{
\begin{array}{rcll}
-\Delta u&=&\lambda f(u)&\textrm{in }\Omega,\\
u&=&0&\textrm{on }\partial \Omega,
\end{array}
\right. \leqno{(2.13)_{\lambda}}
$$
where $\lambda>0$ is a parameter and $f$ is a $C^1$ positive increasing function 
defined on $[0,\infty)$ (in particular $f(0)>0$) which is superlinear at 
infinity (\textit{i.e.}, satisfying $f(t)/t\rightarrow+\infty$ as 
$t\rightarrow+\infty$). Under these assumptions (see the excellent monograph 
\cite{Dupaigne} for all these questions),  
there exists an extremal parameter $\lambda^\star\in(0,\infty)$ such that problem 
$\Plambda$ admits a classical minimal solution $u_\lambda$ 
for $\lambda\in(0,\lambda^\star)$ and admits no weak solution (see Definition \ref{Defn:weak_sol})
for 
$\lambda>\lambda^\star$.
By minimality it is easy to show that $u_\lambda$ is 
a semi-stable solution for $\lambda\in(0,\lambda^\star)$. Moreover, 
$$
u^\star:=\lim_{\lambda\uparrow\lambda^\star}u_\lambda
$$
is a weak solution of $\Pext$, known as the \textit{extremal solution}. Thus,
$u^\star$ is a semi-stable weak solution of $\Pext$.

In full generality (\textit{i.e.}, for all domains $\Omega$ and 
all nonlinearities $f$), the optimal regularity for $u^\star$ 
remains still as open problem. For instance, it is unknown if $u^\star$ always lies in 
the energy class $H^1_0(\Omega)$, or if it is always bounded when $n\leq 9$ 
(see open problems~1 and 4 in \cite{BV}). These questions have a positive 
answer in the radial case for all nonlinearities (see Remark~\ref{regradial} below), and also for 
general domains and power or exponential type nonlinearities.
The optimal $L^q$ and $W^{1,p}$ regularity (depending on the dimension) 
in the general case is also still unknown.

Nedev \cite{Nedev} proved in the case of convex nonlinearities that $u^\star
\in L^\infty(\Omega)$ when $n\leq 3$ and $u^\star\in L^q(\Omega)$ for all $q<n/(n-4)$ 
when $n\geq 4$. Note that these regularity results hold for arbitrary 
smooth domains $\Omega$. In another paper, Nedev \cite{Nedev01} also proved that if in addition 
$\Omega$ is strictly convex then $u^\star\in H^1_0(\Omega)$. 
In particular, $u^\star\in L^\frac{2n}{n-2}(\Omega)$. This is the content 
of the unpublished preprint \cite{Nedev01}. In the present paper, we supply with 
detailed proofs (slightly modified) of the result in \cite{Nedev01} ---see
Theorem~\ref{Thm:Nedev}, Remark~\ref{Rmk:Nedev}, and subsection~\ref{subsection4:3} below.

As in Theorem~\ref{Theorem}, it is also possible to prove that 
$u^\star\in W^{1,p}_0(\Omega)$ for all $p<4n/(3n-4)$. However, 
as we said before, the following $W^{1,2}=H^1$ estimate of Nedev \cite{Nedev01}
---proved using a different argument than ours--- is better than 
the one of Theorem~\ref{Theorem}.

\begin{theorem}[Nedev \cite{Nedev01}]\label{Thm:Nedev}
Let $f:[0,+\infty)\longrightarrow\mathbb{R}$ be an increasing positive $C^1$ 
function such that $f(t)/t\rightarrow+\infty$ as $t\rightarrow+\infty$. Assume 
that $\Omega$ is a convex smooth bounded domain of $\mathbb{R}^n$ with $n\geq 2$.
Then $u^\star \in H^1_0(\Omega)$. In particular, $u^\star\in L^\frac{2n}{n-2}(\Omega)$.
\end{theorem}

To prove Theorem~\ref{Thm:Nedev}, a Poho${\rm\check{z}}$aev
identity and the minimality of $u_\lambda$ is used to obtain
$$
\int_\Omega |\nabla u_\lambda|^2\ dx
\leq 
\frac{1}{2}\int_{\partial\Omega}|\nabla u_\lambda|^2 \ (x\cdot\nu(x))\ d\sigma
\quad\textrm{for all }\lambda\in(0,\lambda^\star),
$$
where $\nu$ is the outward unit normal to $\Omega$.
Then, since $\Omega$ is convex, the moving planes method allows to control the right hand side 
of the previous inequality by $\|u^\star\|_{L^1(\Omega)}$. Since $u^\star$ is a weak 
solution of $\Pext$, and hence $u^\star\in L^1(\Omega)$, Theorem~\ref{Thm:Nedev} 
follows. For the sake of completeness we will prove this result simplifying slightly the 
original proof of Nedev. 	

\begin{remark}\label{Rmk:Nedev}
Nedev \cite{Nedev01} pointed out that Theorem~\ref{Thm:Nedev} also holds for certain nonconvex domains. 
More precisely, let $\nu$ be the outward unit normal to $\Omega$ and 
$
E:=\{x\in\partial\Omega: \textrm{ there exists }$ $\varepsilon>0\textrm{ and 
a hyperplane }P \textrm{ such that }P\cap\overline{\Omega}\cap B_\varepsilon(x)
=\{x\}\}.
$
If there exists $a\in\mathbb{R}^n$ and $\alpha<0$ such that $(x-a)\cdot \nu(x)\leq \alpha$ for 
every $x\in\partial\Omega \setminus E$, then the statement of Theorem~\ref{Thm:Nedev} holds in $\Omega$.
Note that this assumption is satisfied by strictly convex domains, annulus, or bean pea 
shaped domains, for example. See also Remark~\ref{Rmk:Nedev2} below.
\end{remark}

\begin{remark}[{\it Regularity in the radial case}\rm]\label{regradial}
In \cite{CC06} it is studied the regularity of semi-stable radially symmetric 
solutions when the domain is a ball. It is proved that every semi-stable solution, 
in particular the extremal solution of $\Plambda$, is bounded if the dimension 
$n\leq 9$. For $n\geq 10$, it is proved that such a solution belongs to $W^{1,q}_0(B_1)$ 
for all $1\leq q<q_1$, where 
$$
q_1:=\frac{2n}{n-2\sqrt{n-1}-2}.
$$
In particular, it belongs to $L^q(B_1)$ for all $1\leq q<q_0$, where
$$
q_0:=\frac{2n}{n-2\sqrt{n-1}-4}.
$$
It can be shown that these regularity results are sharp by taking explicit semi-stable 
solutions corresponding to the exponential and power nonlinearities.

Note that the $L^{\frac{2n}{n-4}}(\Omega)$ estimate obtained in Theorem~\ref{Theorem2} 
differs from the sharp exponent $q_0$ defined above by the term $2\sqrt{n-1}$. 
\end{remark}

\subsection{Plan of the paper}\label{subsection1:4}

The paper is organized as follows. In section~\ref{section2}, we prove 
the geometric-type inequalities stated in Theorem~\ref{Theorem:Sobolev}. 
In section~\ref{section3}, we deal with semi-stable solutions and we prove the 
estimates stated in Theorems~\ref{Theorem2}~and~\ref{Theorem}. Finally, we prove 
Theorem~\ref{Thm:Nedev} due to Nedev \cite{Nedev01} in an unpublished preprint. 

\section{Geometric-type  Sobolev inequalities. Proof of Theorem~\ref{Theorem:Sobolev}}\label{section2}

The main purpose of this section is to establish Theorem~\ref{Theorem:Sobolev}. 
Its proof uses two isoperimetric inequalities. 
The first one is a  consequence of the Fleming-Rishel 
formula \cite{FleRi} and the classical isoperimetric inequality.
If $v\in W^{1,1}_0(\Omega)$, then
\begin{equation}\label{talenti}
A_1 V(t)^{(n-1)/n}\leq P(t)=\frac{d}{dt}\int_{\{|v|\leq t\}}|\nabla
v|dx \qquad\text{for a.e. }t>0,
\end{equation}
where $A_1:=n|B_1|^{1/n}$, $V(t):=|\{x\in\Omega:|v(x)|>t\}|$, and $P(t)$
stands for the perimeter in the sense of De~Giorgi,
\textit{i.e.}, $P(t)$ is the total variation of the characteristic
function of $\{x\in \Omega: |v(x)|>t\}$. A proof of this inequality can be found in
\cite{Talenti}. We also note that the distribution function 
$V(t)$ is differentiable almost everywhere since it is a nonincreasing 
function.

The second isoperimetric inequality that we use is inequality \eqref{isop:mean},
due to Michael and Simon~\cite{MS} 
and to Allard \cite{A} ---see also Theorem~28.4.1 \cite{Burago}.
We apply it to almost all level sets 
of $|v|$, where $v\in C_0^\infty(\overline{\Omega})$. We have
\begin{equation}\label{mean:estimate0}
P(t)^{\frac{n-2}{n-1}} \leq A_2\int_{\{|v|=t\}\cap\{|\nabla v|>0\}} |H_v|
\ d\sigma\quad\textrm{ for a.e. }t>0.
\end{equation}
Here, $H_v$ is the mean curvature of $\{|v|=t\}$ and $A_2$ is a constant 
depending only on the dimension $n\geq 2$. 
Note that, by Sard's theorem, almost every $t\in(0,\|v\|_{L^\infty(\Omega)})$ is a regular value of 
$|v|$. By definition, if $t$ is a regular value of $|v|$, then 
$\left|\nabla v(x)\right|>0$ for all $x\in\Omega$ such that 
$|v(x)|=t$. In particular, if $t$ is a regular value, $S_t
:=\{x\in\Omega: |v(x)|=t\}$ is a $C^{\infty}$ immersed $(n-1)$-dimensional 
compact hypersurface of $\mathbb{R}^n$ without boundary. Hence, 
we can apply inequality \eqref{isop:mean} to $S=S_t$ obtaining 
\eqref{mean:estimate0}. Note here that, since 
$S$ could have a finite number of connected components, 
inequality \eqref{isop:mean} (and \eqref{mean:estimate0}) for 
connected manifolds $S$ leads to the same inequality (with same 
constant) for $S$ with more than one component. 

{From} \eqref{mean:estimate0} and Jensen inequality, we deduce 
\begin{equation}\label{mean:estimate}
P(t)^{\frac{n-(1+r)}{n-1}} \leq A_2^r\int_{\{|v|=t\}\cap\{|\nabla v|>0\}} 
|H_v|^r\ d\sigma\quad\textrm{for all }r\geq 1.
\end{equation}
Since we always have $n\geq1+r$ in Theorem~\ref{Theorem:Sobolev}, we can now use 
the isoperimetric inequality \eqref{talenti} to conclude
\begin{equation}\label{mean:estimate2}
A_1^{\frac{n-(1+r)}{n-1}}V(t)^{\frac{n-(1+r)}{n}} \leq 
A_2^r\int_{\{|v|=t\}\cap\{|\nabla v|>0\}} |H_v|^r\ d\sigma\quad\textrm{for all }r\geq 1.
\end{equation}

This is the key inequality to prove Theorem~\ref{Theorem:Sobolev}. Note that 
in the case $r=0$, inequality \eqref{mean:estimate2} also holds ---it is nothing 
but the classical isoperimetric inequality \eqref{talenti}. We start by proving 
parts $(a)$ and $(c)$.
\begin{proof}[Proof of Theorem~{\rm\ref{Theorem:Sobolev}} $(a)$ and $(c)$]
First, we deal with the case $p=1$ and $r=n-1$. Integrating \eqref{mean:estimate}
from $0$ to $\|v\|_{L^\infty(\Omega)}$ and using the coarea formula, we obtain
$$
\|v\|_{L^\infty(\Omega)} \leq A_2^{n-1}\int_{\Omega\cap\{|\nabla v|>0\}} |H_v|^{n-1}|\nabla v|\ dx,
$$
\textit{i.e.}, \eqref{Sobolev:infinity} with $C_1=A_2^{n-1}$.

Assume now $p>1$ and $r\in [1,n-1]$. Using the coarea 
formula and that almost every $t\in(0,\|v\|_{L^\infty(\Omega)})$ is a regular value 
of $|v|$, we have
$$
-V'(t)=\int_{\{|v|=t\}\cap\{|\nabla v|>0\}}\frac{d\sigma}{|\nabla v|}\qquad\textrm{for a.e. }
t>0.
$$
Hence, by \eqref{mean:estimate2} and H\"older inequality we obtain
$$
A_1^{\frac{n-(1+r)}{n-1}}A_2^{-r}\ V(t)^\frac{n-(1+r)}{n}
\leq (-V'(t))^{1/p'}\left(\int_{\{|v|=t\}\cap\{|\nabla v|>0\}}
\hspace{-0.5cm}|H_v|^{pr}|\nabla v|^{p-1}\ d\sigma\right)^{1/p}
$$
for a.e. $t>0$, where $p'=p/(p-1)$, or equivalently,
$$
1\leq A
\left(V(t)^{-\frac{n-(1+r)}{n}p'}(-V'(t))\right)^{1/p'}\ 
\left(\int_{\{|v|=t\}\cap\{|\nabla v|>0\}}|H_v|^{pr}|\nabla v|^{p-1}\ d\sigma\right)^{1/p}
$$
for a.e. $t>0$ such that $V(t)>0$, where $A=A_1^{-\frac{n-(1+r)}{n-1}}A_2^r$. 
Integrating the previous inequality with respect to $t$ in $(0,s)$ 
and using H\"older inequality, we have 
\begin{equation}\label{eq1:prop1}
s\leq A\left(\int_{V(s)}^{|\Omega|} 
\tau^{-\frac{n-(1+r)}{n}p'}\ d\tau\right)^{1/p'}
\left(\int_{\Omega\cap\{|\nabla v|>0\}}|H_v|^{pr}|\nabla v|^p\ dx\right)^{1/p}
\end{equation}
for a.e. $s\in(0,\|v\|_{L^\infty(\Omega)})$. Let $$\beta:=-\frac{n-(1+r)}{n}p'+1=-\frac{n-p(1+r)}{(p-1)n}.$$

$(a)$ Assume $n<p(1+r)$ and note that $\beta>0$. Therefore, letting 
$s\uparrow\|v\|_{L^\infty(\Omega)}$ in \eqref{eq1:prop1}, we obtain
$$
\|v\|_{L^\infty(\Omega)}\leq \frac{A|\Omega|^{\frac{p(1+r)-n}{np}}}{\beta^{1/p'}}
\left(\int_{\Omega\cap\{|\nabla v|>0\}}|H_v|^{pr}|\nabla v|^p\ dx\right)^{1/p},
$$
proving the remaining case of assertion $(a)$.

$(c)$ Assume $n=p(1+r)$ and $p>1$. From \eqref{eq1:prop1}, we obtain 
$$
s\leq A\left(\int_{V(s)}^{|\Omega|}\frac{d\tau}{\tau}\right)^{1/p'}
\left(\int_{\Omega\cap\{|\nabla v|>0\}}|H_v|^{pr}|\nabla v|^p\ dx\right)^{1/p}
\textrm{ for a.e. }s\in(0,\|v\|_{L^\infty(\Omega)}),
$$
and therefore, 
\begin{equation}\label{quasilinear:b1}
V(s)\leq |\Omega|\exp\left\{-\left(\frac{s}{A I_p}\right)^{p'}\right\}\qquad 
\textrm{ for a.e. }s\in(0,\|v\|_{L^\infty(\Omega)}),
\end{equation}
where $I_p:=\left(\int_{\Omega\cap\{|\nabla v|>0\}}|H_v|^{pr}|\nabla v|^p\ dx\right)^{1/p}$. 
Let $k$ be any positive integer. Using \eqref{quasilinear:b1} we obtain
$$
\begin{array}{lll}
\displaystyle\int_\Omega|v|^{kp'}\ dx&=&\displaystyle kp'\int_0^\infty s^{kp'-1}V(s)\ ds\\
&\leq&\displaystyle kp'|\Omega|\int_0^\infty s^{kp'-1}e^{-\left(\frac{s}{AI_p}\right)^{p'}}\ ds\\
&=&\displaystyle k|\Omega|(AI_p)^{kp'}\int_0^\infty \tau^{k-1}e^{-\tau}\ d\tau\vspace{0.2cm}\\
&=&\displaystyle |\Omega|(AI_p)^{kp'}k!.
\end{array}
$$
Let $C_3>A$ (remember that here $A$ depends only on $n$ and $p$ since 
$r=(n-p)/p$) 
be any positive constant. Then, the previous inequality leads to
$$
\begin{array}{lll}
\displaystyle \int_\Omega\exp\left\{\left(\frac{|v|}{C_3 I_p}\right)^{p'}\right\}\ dx
&\leq&\displaystyle \sum_{k=0}^\infty\left(\frac{A}{C_3}\right)^{kp'}|\Omega|
=\frac{C_3^{p'}}{C_3^{p'}-A^{p'}}|\Omega|.
\end{array}
$$
This ends the proof of parts (a) and (c) of Theorem~\ref{Theorem:Sobolev}.
\end{proof}

\begin{remark}\label{Remark:Sobolev}
In the previous proof we have obtained the following explicit expressions for
the constants in parts $(a)$ and $(c)$ of Theorem~\ref{Theorem:Sobolev}.
Here, $A_1=n|B_1|^{1/n}$ and $A_2$ denote 
the constants appearing in \eqref{isopctant} and \eqref{isop:mean}, 
respectively, which depend only on $n$.

The constant in the $L^\infty$ estimate of part $(a)$ can be taken to be
$$
C_1=\left(\frac{(p-1)n}{p(1+r)-n}\right)^{1-\frac{1}{p}}A_1^{\frac{1+r-n}{n-1}} A_2^r 
$$
when $p>1$, and $C_1=A_2^{n-1}$ when $p=1$ and $r=n-1$. Trudinger's type 
inequality \eqref{quasilinear:b2} holds for all 
$$
C_3>A_1^{-\frac{n}{(n-1)p'}}A_2^{\frac{n}{p}-1}=:A
$$
and the constant $C_4$ is given by $C_4=C_3^{p'}/(C_3^{p'}-A^{p'})$. 
\end{remark}

\begin{remark}
Assume $p>1$ and $n>p(1+r)$. Let $p_r^\star$ the critical Sobolev exponent defined in 
Theorem~\ref{Theorem:Sobolev} $(b)$. Computing the first integral in 
\eqref{eq1:prop1}, we deduce 
$$
V(s)\leq |\Omega|\left(\frac{p'}{p_r^\star}\left(\frac{|\Omega|^{1/p_r^\star}}{AI_{p}}\right)^{p'} s^{p'}+1\right)^{-p_r^\star/p'}\quad\textrm{ for a.e. }s\in(0,\|v\|_{L^\infty(\Omega)}),
$$
where 
$I_{p}:=\left(\int_{\Omega\cap\{|\nabla v|>0\}} |H_v|^{pr}|\nabla v|^p\ dx\right)^{1/p}$. Noting that
$$
\int_\Omega|v|^q\ dx=q\int_0^\infty s^{q-1}V(s)\ ds,
$$
one obtains that, for some constant $C$ depending only on 
$n$, $p$, $r$, and $q$,
$$
\left(\int_{\Omega}|v|^{q}\ dx\right)^{1/q}
\leq C|\Omega|^{\frac{1}{q}-\frac{1}{p_r^\star}}
\left(\int_{\Omega\cap\{|\nabla v|>0\}} |H_v|^{pr}|\nabla v|^p\ dx\right)^{1/p}
$$
for all $q<p_r^\star$. The constant $C$ may be chosen to be
$$
C=\left(\frac{q}{p'}\right)^\frac{1}{q}\left(\frac{p'}{p_r^\star}\right)^{-\frac{1}{p'}}A
\left(\int_0^\infty \tau^{\frac{q}{p'}-1}(\tau+1)^{-\frac{p_r^\star}{p'}}\ d\tau\right)^{1/q},$$
which is finite if and only if $q<p_r^\star$. However, using this 
argument it is not possible to obtain the inequality with the critical 
Sobolev exponent $q=p^\star_r$. Although we could introduce Schwarz (or decreasing) 
symmetrization 
in order to get the critical exponent $p_r^\star$, we use the following slightly 
different argument.
\end{remark}

Now, we prove Theorem~\ref{Theorem:Sobolev} $(b)$.
\begin{proof}[Proof of Theorem~{\rm\ref{Theorem:Sobolev}} $(b)$]
Assume $n>p(1+r)$ and let $p_r^\star=np/(n-p(1+r))$. 
Integrating \eqref{mean:estimate2} from 0 to $M:=\|v\|_{L^\infty(\Omega)}$, 
we obtain
\begin{equation}\label{mean:estimate3}
A_1^{\frac{n-(1+r)}{n-1}}\int_0^MV(t)^{\frac{n-(1+r)}{n}} \ dt\leq 
A_2^r\int_{\Omega\cap\{|\nabla v|>0\}} |H_v|^r|\nabla v|\ dx\quad\textrm{for all }r\geq 1.
\end{equation}

Let 
$$
W(t):=\left(\int_0^t V(s)^{\frac{n-(1+r)}{n}}\ ds\right)^{1_r^\star}.
$$
Using that $V(t)$ is a nonincreasing function, we easily deduce
$$
1_r^\star \ t^{1_r^\star-1}V(t)\leq W'(t)\quad \textrm{for a.e. }t\in(0,M).
$$
Hence, integrating from 0 to $M$, we get
$$
\int_{\Omega}|v|^{1_r^\star}\ dx
=1_r^\star\int_0^M t^{1_r^\star-1}V(t)\ dt
\leq 
W(M)=\left(\int_0^M V(t)^{\frac{n-(1+r)}{n}}\ dt\right)^{1_r^\star}.
$$
Combining this with \eqref{mean:estimate3} we obtain
\begin{equation}\label{mean:estimate:aux}
A_1^{\frac{n-(1+r)}{n-1}}\left(\int_{\Omega}|v|^{1_r^\star}\ dx\right)^{1/1_r^\star}
\leq A_2^r\int_{\Omega\cap\{|\nabla v|>0\}} |H_v|^r|\nabla v| \ dx,
\end{equation}
\textit{i.e.}, assertion $(b)$ for $p=1$.
 
For $p>1$, we only need to apply inequality \eqref{mean:estimate:aux} 
with $|v|$ replaced by $|v|^\gamma$ and $\gamma=p_r^\star/1_r^\star$
(noting that the level sets of $|v|$ and $|v|^\gamma$ 
are the same, and hence, their mean curvatures coincide)
and use H\"older inequality to conclude \eqref{Sobolev}.
\end{proof}

\begin{remark}\label{Remark:Sobolev2}
Inequality \eqref{Sobolev} in Theorem~\ref{Theorem:Sobolev} $(b)$ holds with the constant
$$
C_2=\frac{n-(1+r)}{n-p(1+r)}\ p\ \Big(n|B_1|^{1/n}\Big)^{-\frac{n-(1+r)}{n-1}}
A_2^r,
$$
where $A_2$ is the constant appearing in \eqref{isop:mean}. 
\end{remark}

\begin{remark}\label{stadium}
Here we show that the inequalities in Theorem~\ref{Theorem:Sobolev} 
(both the Sobolev and the Morrey inequalities) do not hold
when $r\in(0,2p^{-1}-1)$ and $1\leq p<2$. In particular, they do not hold for $r\in(0,1)$
and $p=1$.

We also study the geometric inequalities behind \eqref{mean:estimate} 
and \eqref{mean:estimate2}. That is, we study the inequalities
\begin{equation}\label{ineq:aux1}
|\partial\Omega|^\frac{n-(1+pr)}{n-1}\leq C \int_{\partial\Omega}|H|^{pr}\ d\sigma
\end{equation}
and
\begin{equation}\label{ineq:aux2}
|\Omega|^\frac{n-(1+pr)}{n}\leq C \int_{\partial\Omega}|H|^{pr}\ d\sigma,
\end{equation}
where $H$ is the mean curvature of $\partial\Omega\subset\mathbb{R}^n$. 
We show that for every constant $C=C(n,p,r)$ inequalities \eqref{ineq:aux1}
and \eqref{ineq:aux2} fail, even among convex sets $\Omega\subset\mathbb{R}^n$, 
when $r\in(0,1/p)$ and $p\geq 1$.

\begin{figure}[ht]
\begin{center}
    \includegraphics[totalheight=7cm]{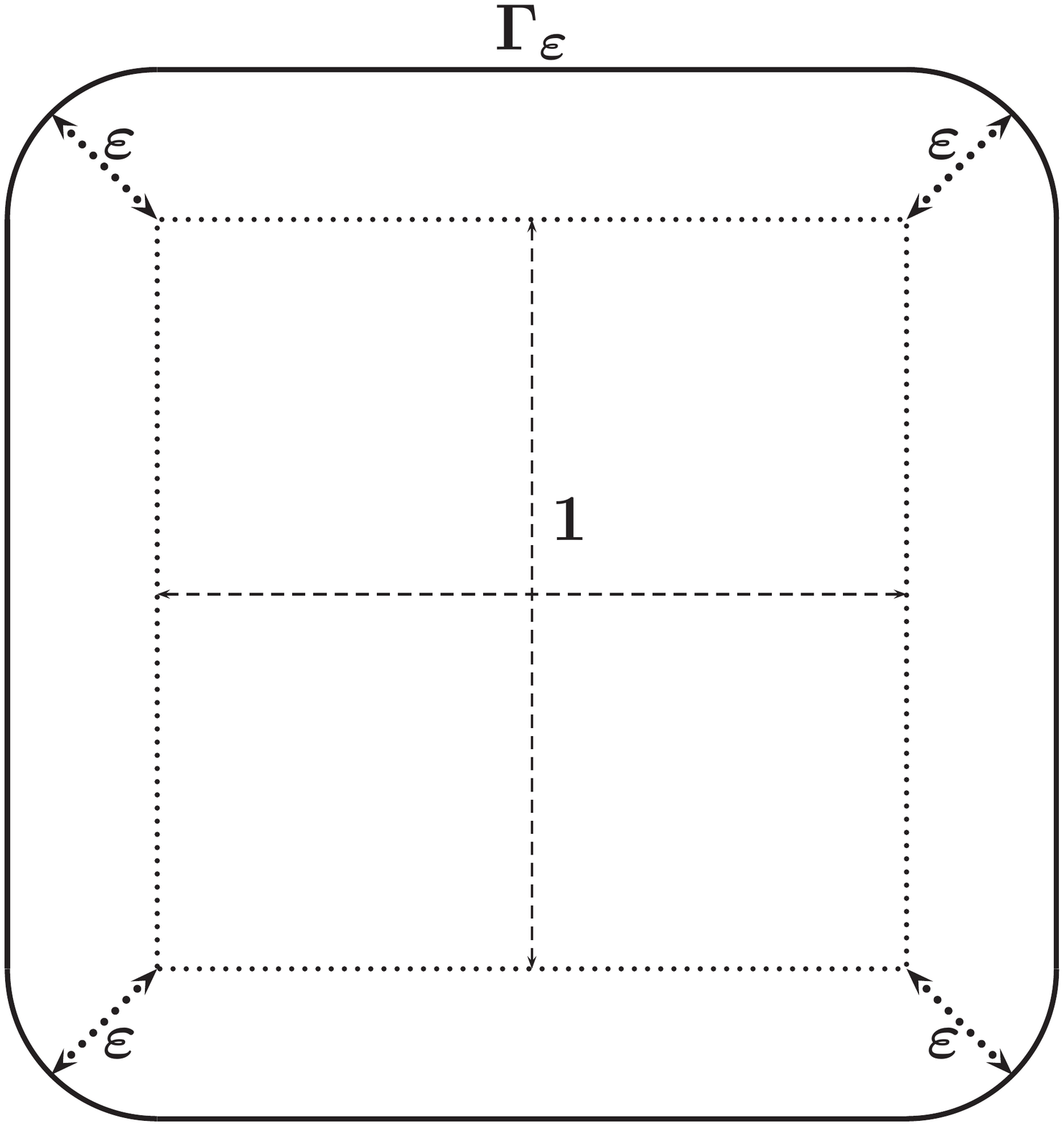}
\end{center}
  \caption{Level sets of $v$.}
  \label{fig:stadium}
\end{figure}

To see all this, let $Q_1=(0,1)^n$ be the open unitary cube of $\mathbb{R}^n$, $n\geq 2$.
Given $\varepsilon\in(0,1/2)$, set $\Gamma_\varepsilon:=
\{x\in\mathbb{R}^n\setminus Q_1:{\rm dist}(x,Q_1)=\varepsilon\}$
and $\Omega_\varepsilon$ to be its bounded interior.
Let $H_{\Gamma_\varepsilon}$ be the mean curvature of $\Gamma_\varepsilon$ and  
$A_\varepsilon:=\{x\in\Gamma_\varepsilon: H_{\Gamma_\varepsilon}(x)
\neq 0\}$. Note that
\begin{equation}\label{ex:aux1}
H_{\Gamma_\varepsilon}\equiv 0\textrm{ on }\Gamma_\varepsilon\setminus A_\varepsilon,
\quad |A_\varepsilon|\leq c_1 \varepsilon,\quad\textrm{ and }\quad 
|H_{\Gamma_\varepsilon}|\leq c_2 \varepsilon^{-1}\textrm{ on }A_\varepsilon,
\end{equation}
where $c_1$ and $c_2$ are constants depending only on $n$.
Therefore, since $r>0$, 
\begin{equation}\label{cube2}
\int_{\Gamma_\varepsilon}|H_{\Gamma_\varepsilon}|^{pr}\ d\sigma
\leq 
c_3\varepsilon^{1-pr},
\end{equation}
where $c_3$ is a constant depending only on $n$, $p$, and $r$.
Since $|\Gamma_\varepsilon|>1$ and $|\Omega_\varepsilon|>1$ 
for all $\varepsilon\in (0,1/2)$, and the right hand side of 
\eqref{cube2} tends to zero, as $\varepsilon$ goes to zero, when 
$r\in (0,1/p)$, we obtain that \eqref{ineq:aux1} and \eqref{ineq:aux2}
do not hold for $r\in(0,1/p)$, as claimed.

Although $\Gamma_\varepsilon$ is not $C^\infty$ (since ${\rm dist}(\cdot,Q_1)$ 
is not a $C^\infty$ function), the same facts hold for $C^\infty$ immersed 
$(n-1)$-dimensional compact hypersurfaces of $\mathbb{R}^n$. Indeed, there exists $\tilde{d}\in C_c^\infty(\mathbb{R}^n)$ 
such that 
$0 \leq \tilde{d} \leq 1$, 
$\tilde{d}\equiv 0$ in $Q_1$, 
$\tilde{d}\equiv 1$ in $\{x\in\mathbb{R}^n:{\rm dist}(x,Q_1)\geq1\}$,
$|\nabla\tilde{d}|\leq 2$, 
and its level sets 
$\tilde{\Gamma}_\varepsilon:=\{x\in\mathbb{R}^n: \tilde{d}(x)=\varepsilon\}$,
$0<\varepsilon<1$, satisfy \eqref{ex:aux1} (and hence \eqref{cube2}). This can be 
seen choosing a hypersurface $\tilde{\Gamma}_1$ coinciding with $\Gamma_1$ in 
its flat parts and smoothing it in vertex, edges, etc. Then we define
$\tilde{\Gamma}_\varepsilon$ for $0<\varepsilon<1$ as an homotethy with respect 
to the vertex, edges, etc., of the cube. In this way, $\tilde{\Gamma}_\varepsilon$ produces 
a foliation of $\{x\in\mathbb{R}^n:0<{\rm dist}(x,Q_1)<1\}$. We finally define 
$\tilde{d}(x)=\varepsilon$ if and only if $x\in\tilde{\Gamma}_\varepsilon$.

Now, we can prove that the inequalities in Theorem~\ref{Theorem:Sobolev} fail
whenever $r\in(0,2p^{-1}-1)$.
Let $v$ be a positive function whose level sets are $\tilde{\Gamma}_\varepsilon$. 
More precisely,  let $\psi:[0,+\infty)\longrightarrow\mathbb{R}$ 
be any decreasing $C^\infty$ function such that $\psi(s)=0$ for $s\geq 1$ and 
$\psi^{i)}(0)=0$ for all $i\geq 1$. 
Given $\varepsilon_0\in(0,1)$, we define $v(x)=\psi(\tilde{d}(x)/\varepsilon_0)$ 
in $\mathbb{R}^n\setminus[0,1]^n$ and $v=\psi(0)$ in $[0,1]^n$. Note that 
$v \in C^\infty(\overline{\Omega})$ where $\Omega$
is the bounded interior of $\tilde{\Gamma}_{\varepsilon_0}$.

Using the coarea formula, $|\nabla \tilde{d}|\leq 2$, \eqref{cube2}, 
and the change of variables $v=\psi(\tilde{d}/\varepsilon_0)=t=\psi(s)$,
it is easy to check that
$$
\begin{array}{lll}
\displaystyle \int_{\Omega\cap\{|\nabla v|>0\}}|H_v|^{pr}|\nabla v|^p\ dx
&=& \displaystyle 
\int_0^{\psi(0)} \int_{\{v=t\}\cap \{|\nabla v|>0\}}|H_v|^{pr}|\nabla v|^{p-1}\ d\sigma\ dt
\\
&\leq&
\displaystyle 
C\int_0^1 |\psi'(s)|^p\varepsilon_0^{-(p-1)}
\int_{\tilde{\Gamma}_{\varepsilon_0s}} |H_{\varepsilon_0 s}|^{pr}\ d\sigma\ ds\\
&\leq&
\displaystyle 
\frac{C}{\varepsilon_0^{p-1}}\varepsilon_0^{1-pr}
\int_0^1 
|\psi'(s)|^p\ s^{1-pr}\ ds,
\end{array}
$$
where $C$ is a constant depending only on $n$, $p$, and $r$.
Note that the right hand side of this inequality tends to zero 
as $\varepsilon_0$ goes to $0$ if $r<2p^{-1}-1$. On the other hand, it is clear 
that, for any $q\geq 1$,  
$$
\|v\|_{L^q(\Omega)}\geq \left(\int_{Q_1}|v|^q\ dx\right)^{1/q}
=\psi(0)=\|v\|_{L^\infty(\Omega)}>0.
$$ 
Therefore, a necessary condition in order that Theorem~\ref{Theorem:Sobolev} 
holds (in the range $r>0$) is $r\geq 2p^{-1}-1$. In particular, if $p=1$ the necessary 
condition is that $r\geq1$.
\end{remark}

\begin{remark}\label{Rmk:perimeter}
We derive two more inequalities involving the perimeter $P(t)$
of the level sets. On the one hand, using \eqref{mean:estimate}, integrating 
with respect to $t$ in $(0,\|v\|_{L^\infty(\Omega)})$, and using the coarea formula, 
we obtain 
$$
\|v\|_{L^\infty(\Omega)}\leq A_2^r\int_{\Omega\cap\{|\nabla v|>0\}} 
P(v)^{\frac{1+r-n}{n-1}}|H_v|^r|\nabla v| \ dx\quad
\textrm{for all } n\geq 2,\ r\geq 1.
$$

On the other hand, note that the total variation of $v$ may 
be written as 
$$
\int_\Omega|\nabla v|\ dx=\int_0^{\|v\|_{L^\infty(\Omega)}} P(t)\ dt,
$$
and that by \eqref{mean:estimate} we have
$$
1\leq A_2^{n-1}P(t)^\frac{1+r-n}{r}\left(\int_{\{|v|=t\}\cap\{|\nabla v|>0\}}
|H_v|^r\ d\sigma\right)^\frac{n-1}{r}.
$$
In the case $2\leq n<1+r$ (which is not considered in Theorem~\ref{Theorem:Sobolev}), 
integrating the previous inequality with respect to $t$ in $(0,\|v\|_{L^\infty(\Omega)})$ and using 
H\"older inequality, we obtain 
$$
\|v\|_{L^\infty(\Omega)}
\leq
A_2^{n-1}
\left(\int_\Omega|\nabla v|\ dx\right)^{\frac{1+r-n}{r}}
\left(\int_{\Omega\cap\{|\nabla v|>0\}}|H_v|^r|\nabla v|\ dx\right)^{\frac{n-1}{r}}.
$$
\end{remark}
\section{Semi-stable solutions. Proof of Theorems~\ref{Theorem}, 
\ref{Theorem2}, and \ref{Thm:Nedev}}\label{section3}
This section deals with semi-stable solutions. We apply Theorem~\ref{Theorem:Sobolev} 
to prove Theorems~\ref{Theorem} and \ref{Theorem2}. Finally we prove Theorem~\ref{Thm:Nedev}
using a Poho${\rm\check{z}}$aev identity and the fact that the extremal solution $u^\star$ 
is the increasing limit in $L^1$ of minimal classical solutions. 

To obtain the $L^\frac{2n}{n-4}$ estimate of Theorems~\ref{Theorem} and \ref{Theorem2} 
we use the semi-stability 
condition \eqref{semi-stab1} with test function $\xi=|\nabla u|\eta$, 
where $u$ is a smooth semi-stable solution of \eqref{problem} and $\eta$ vanishes 
on $\partial\Omega$ and is still arbitrary. With this choice one has
\begin{eqnarray}
\displaystyle \int_{\Omega\cap\left\{\left|\nabla u\right|>0\right\}} 
|B_u|^2|\nabla u|^2\eta^2\ dx
&\leq&\displaystyle \int_{\Omega\cap\left\{\left|\nabla u\right|>0\right\}} 
\left( |\nabla_T |\nabla u||^2 +|B_u|^2|\nabla u|^2\right)\eta^2\ dx \nonumber
\\
&\leq& \displaystyle \int_\Omega |\nabla u|^2 |\nabla \eta|^2 \ dx, \label{neweqn}
\end{eqnarray}
for every Lipschitz function $\eta$ in $\overline\Omega$ with 
$\eta|_{\partial\Omega}\equiv0$ (see for instance Proposition 2.2 of \cite{Cabre09} 
and references therein). 
Here, $\nabla_T$ denotes the 
tangential gradient along a level set of $|u|$ and 
$$
|B_u(x)|^2=\sum_{i=1}^{n-1} \kappa_i^2(x),
$$
where $\kappa_i(x)$ are the principal curvatures of the level set of $|u|$ passing
through $x$, for a given $x\in\Omega\cap\left\{\left|\nabla u\right|>0\right\}$.
Now, noting that $(n-1)H_u^2\leq |B_u|^2$, we deduce inequality \eqref{semi1} 
from \eqref{neweqn}:
\begin{equation}\label{semi1:new}
(n-1)\int_{\Omega\cap\{|\nabla u|>0\}} H_u^2|\nabla u|^2\eta^2\ dx
\leq \int_\Omega |\nabla u|^2|\nabla\eta|^2 \ dx.
\end{equation}

\subsection{Proof of Theorem~\ref{Theorem}}\label{subsection4:1}
The $L^\frac{2n}{n-4}$ estimate will follow from \eqref{semi1:new}. Instead, the $W^{1,p}$ 
estimates of Theorem~\ref{Theorem} will use the following result. It holds for solutions 
of the linear problem
\begin{equation}\label{linear}
\left\{
\begin{array}{rcll}
-\Delta u&=&h(x)&\textrm{in }\Omega,\\
u&=&0&\textrm{on }\partial \Omega.
\end{array}
\right.
\end{equation}

\begin{proposition}\label{Thm:bootstrap}
Assume $n\geq 3$ and $h\in L^1(\Omega)$. If $u\in 
W^{1,1}_0(\Omega)\cap L^q(\Omega)$ is a solution $($in 
the distributional sense$)$ of \eqref{linear} for some $q\geq n/(n-2)$, 
then 
%
$$
\int_\Omega |\nabla u|^p\ dx\leq p|\Omega|
+
\left(\frac{p_q}{p}-1\right)^{-1}\big\{\|u\|_{L^q(\Omega)}^q+\|h\|_{L^1(\Omega)}\big\}
$$
%
for all $p< p_q:=\frac{2q}{q+1}$.
\end{proposition}

\begin{remark}
Assume $h\in L^1(\Omega)$. By standard estimates for elliptic equations, 
there exists a constant $C$ depending only on $n$, $p$, and $|\Omega|$, such that
$$
\int_\Omega|\nabla u|^p\ dx\leq C\|h\|_{L^1(\Omega)}\quad\textrm{for all }p<\frac{n}{n-1}
$$
for every solution $u$ of \eqref{linear}.
The critical exponent $p=n/(n-1)$ can not be reached. In Proposition~\ref{Thm:bootstrap},
under the additional assumption 
$u\in L^q(\Omega)$ for some $q\geq n/(n-2)$, we improve the previous estimate; 
note that $p_q:=2q/(q+1)\geq n/(n-1)$. 

The exponent $p_q$ in Proposition \ref{Thm:bootstrap} is the same as the one in the 
Gagliardo-Nirenberg interpolation inequality
$$
\|\nabla u\|_{L^{p_q}(\Omega)}\leq C\|u\|_{W^{2,1}(\Omega)}^{1/2}\|u\|_{L^q(\Omega)}^{1/2}.
$$
Note that in Proposition~\ref{Thm:bootstrap} we assume 
$-\Delta u=h\in L^1(\Omega)$ and $u\in L^q(\Omega)$.
\end{remark}

The proof of Proposition~\ref{Thm:bootstrap} is based in a technique introduced 
by B\'enilan \textit{et al.} \cite{BBGGPV95} to obtain gradient estimates for 
the entropy solution of problem \eqref{linear} with the Laplacian replaced 
by the $p$-Laplacian.

\begin{proof}[Proof of Proposition~{\rm\ref{Thm:bootstrap}}]
Multiplying \eqref{linear} by 
$T_s u=\max\{-s,$ $\min\{s,u\}\}$ we obtain 
$$
\int_{\{|u|\leq s\}}|\nabla u|^2\ dx=\int_\Omega h(x)T_su\ dx\leq s\|h\|_{L^1(\Omega)}.
$$
{F}rom this, we deduce
$$
\begin{array}{lll}
\displaystyle s^{q}|\{|\nabla u|>s^{(q+1)/2}\}|&\hspace{-0.3cm}\leq&\displaystyle\hspace{-0.3cm} 
s^{q}\int_{\{|\nabla u|>s^{(q+1)/2}\}\cap\{|u|\leq s\}}\left(\frac{|\nabla u|}{s^{(q+1)/2}}\right)^2 dx
+s^{q}\int_{\{|u|>s\}} dx\\
\\
&\hspace{-0.3cm}\leq &\displaystyle \hspace{-0.3cm}\|h\|_{L^1(\Omega)} +s^{q}V(s),\quad\textrm{for a.e. }s>0.
\end{array}
$$
Recall that $V(s)=|\{x\in\Omega:|u(x)|>s\}|$. Letting $t=s^{(q+1)/2}$, we have 
\begin{equation}\label{lalarito}
t^{2q/(q+1)}|\{|\nabla u|>t\}|\leq 
\sup_{\sigma>0}\Big\{\sigma^{q}V(\sigma)\Big\}+\|h\|_{L^1(\Omega)}\ ,
\quad\textrm{for a.e. }t>0.
\end{equation}

Moreover, since
$$
\sigma^{q}V(\sigma)\leq \sigma^q\int_{\{|u|>\sigma\}}\frac{|u|^q}{\sigma^q}\ dx
\leq\int_\Omega|u|^q\ dx=\|u\|_{L^q(\Omega)}^q, \quad\textrm{for a.e. }\sigma>0,
$$
we have $\sup_{\sigma>0}\Big\{\sigma^{q}V(\sigma)\Big\}\leq \|u\|_{L^q(\Omega)}^q$. Therefore, from 
\eqref{lalarito} we deduce
$$
\begin{array}{lll}
\displaystyle \int_\Omega |\nabla u|^p\ dx
&=& \displaystyle p\int_0^\infty t^{p-1}|\{|\nabla u|>t\}|\ dt\\
&\leq& \displaystyle p|\Omega|+p\int_1^\infty t^{p-1}t^{-\frac{2q}{q+1}}\Big(\|u\|_{L^q(\Omega)}^q
+\|h\|_{L^1(\Omega)}\Big)\ dt,
\end{array}
$$
proving the proposition.
\end{proof}

Using \eqref{semi1:new}, Proposition~\ref{Thm:bootstrap}, 
and applying Theorem~\ref{Theorem:Sobolev} $(b)$ to $|u|-s$ with 
$\Omega$ replaced by $\{x\in\Omega: |u(x)|>s\}$, we prove Theorem~\ref{Theorem}.
\begin{proof}[Proof of Theorem~{\rm\ref{Theorem}}]
Since $g\in C^\infty$, we have that $u\in C^\infty(\overline{\Omega})$. Recall that we 
assume $n\geq 5$. By taking $\eta=T_s u=\max\{-s,\min\{s,u\}\}$ in \eqref{semi1:new}, we obtain
\begin{equation}\label{semi2}
(n-1)\int_{\{|u|>s\}\cap\{|\nabla u|>0\}}H_{u}^2|\nabla u|^2\ dx
\leq \frac{1}{s^2}\int_{\{|u|\leq s\}} |\nabla u|^4 \ dx,
\end{equation}
for all $s>0$. We apply Theorem~\ref{Theorem:Sobolev} $(b)$ to 
$v=u-s\in C^\infty(\overline{\Omega})$ with $p=2$ and $r=1$, 
replacing $\Omega$ by each component of $\{x\in\Omega:u(x)>s\}$ (which is $C^\infty$ 
for a.e. $s$). Using also 
\eqref{semi2} we deduce
$$
\begin{array}{lll}
\displaystyle \left(\int_{\{u>s\}}\Big(u-s\Big)^{\frac{2n}{n-4}}\ dx\right)^{\frac{n-4}{2n}}
&\leq&\displaystyle  C_2\left(\int_{\{u>s\}\cap\{|\nabla u|>0\}} H_u^{2}|\nabla u|^2 
\ dx\right)^{\frac{1}{2}}\\
&\leq&\displaystyle  \frac{C(n)}{s}\left(\int_{\{u\leq s\}} |\nabla u|^4 \ dx\right)^{\frac{1}{2}},
\end{array}
$$
for a.e. $s>0$, where $C(n)$ depends only on $n$. Doing the same argument for $-u-s$ 
in $\{-u>s\}$ we conclude \eqref{Lq:estimate}.

Finally, \eqref{grad:estimate} follows applying Proposition~\ref{Thm:bootstrap} 
with $q=2n/(n-4)$.
\end{proof}

\subsection{Proof of Theorem~\ref{Theorem2}}\label{subsection4:2}
To prove Theorem~\ref{Theorem2} we need to control the right hand side of 
\eqref{Lq:estimate}. We accomplish this 
using a boundary regularity result for positive solutions in convex domains. 
More precisely, we use the following result from \cite{GNN, FLN}
(see also \cite{Dupaigne} for its proof).

\begin{proposition}[\cite{GNN, FLN}]\label{Prop4}
Let $f$ be any locally Lipschitz function and let $\Omega$ be a smooth bounded 
domain of $\mathbb{R}^{n}$. Let $u$ be any 
positive classical solution of \eqref{problem}.

If $\Omega$ is convex, then there exist positive constants $\varepsilon$ 
and $\gamma$ depending only on the domain $\Omega$ such that
for every $x\in\Omega$ with $\text{\rm dist}(x,\partial\Omega)<\varepsilon$,
there exists a set $I_x\subset\Omega$ with the following properties: 
$$
|I_x|\geq\gamma \qquad\text{and}\qquad
u(x) \leq u(y) \ \text{ for all }  y\in I_x.
$$
As a consequence,
$$
\Vert u\Vert_{L^\infty(\Omega_\varepsilon)}\leq \frac{1}{\gamma} \Vert u\Vert_{L^1 (\Omega)},
\ \ \text{ where }  \Omega_\varepsilon=\{x\in\Omega\, :\, \text{\rm dist}(x,\partial\Omega)
<\varepsilon\}.
$$
\end{proposition}

We recall (see \cite{Dupaigne}) that it is well known that the extremal solution $u^\star$ 
belongs to $L^1(\Omega)$ and it is a weak solution of $\Pext$ in the 
following sense.
\begin{definition}\label{Defn:weak_sol}
Let $\delta(x):={\rm dist}(x,\partial\Omega)$.
We say that $u\in L^1(\Omega)$ is a \textit{weak solution} of \eqref{problem} if 
$g(u)\delta\in L^1(\Omega)$ and 
$$
\int_\Omega u(-\Delta \varphi)\ dx=\int_\Omega g(u)\varphi\ dx\qquad
\textrm{for all }\varphi\in C^2(\overline{\Omega})\textrm{ with }\varphi|_{\partial\Omega}=0.
$$
\end{definition}

Since $u^\star\in L^1(\Omega)$, from Proposition~\ref{Prop4} we deduce next 
that $u^\star$ is bounded (and smooth) in a neighborhood of the boundary 
if the domain is convex. 
This and Theorem \ref{Theorem} give Theorem~\ref{Theorem2}.

\begin{proof}[Proof of Theorem~{\rm\ref{Theorem2}}]
Assume first that $f\in C^\infty(\mathbb{R})$. Let 
$u_\lambda\in C^\infty(\overline{\Omega})$ be the 
minimal solution of $\Plambda$ for $\lambda\in
(0,\lambda^\star)$.
By Proposition~\ref{Prop4}, and noting that the extremal 
solution $u^\star$ is the increasing limit of $\{u_\lambda\}$, 
there exist constants $\varepsilon$ and $\gamma$ independent of 
$\lambda$ such that
\begin{equation}\label{newnew}
\Vert u_\lambda\Vert_{L^\infty(\Omega_\varepsilon)}
\leq 
\frac{1}{\gamma} \Vert u^\star\Vert_{L^1 (\Omega)}
\qquad\textrm{for all }\lambda<\lambda^\star,
\end{equation}
where  
$$
\Omega_\varepsilon:=\{x\in\Omega : \text{\rm dist}(x,\partial\Omega)<\varepsilon\}.
$$
By taking $\varepsilon$ smaller if necessary, we may assume that
$\Omega_\delta$ is $C^\infty$ for every $0<\delta\leq\varepsilon$.

We can conclude the proof in two ways. First, 
we proceed as in the proof of Proposition 3.1 in \cite{Cabre09}. For this,
note that if $\lambda^\star/2<\lambda<\lambda^\star$, then 
$$ 
u_\lambda\geq
u_{\lambda^\star/2} > c\, \text{\rm dist}(\cdot,\partial\Omega)
$$
for some positive constant $c$ independent of $\lambda\in (\lambda^\star/2,\lambda^\star)$.
Therefore, letting
$$
\tilde{s} :=c\frac{\varepsilon}{2},
$$
we have
$$
\left\{x\in\Omega: u_\lambda(x)\leq \tilde{s}\right\}\subset \Omega_{\varepsilon/2}.
$$
We now use \eqref{Lq:estimate} in Theorem~\ref{Theorem} with $s$ replaced by $\tilde{s}$.
It suffices to bound $\Vert u_\lambda\Vert_{W^{1,4}(\Omega_{\varepsilon/2})}$.
But $u_\lambda$ is a solution of the linear equation $-\Delta u_\lambda= 
h(x):= \lambda f(u_\lambda(x))$ in 
$\Omega_\varepsilon$ and
$u_\lambda=0$ on $\partial\Omega$ (which is one part of $\partial\Omega_\varepsilon$).
On the other hand, $\partial\Omega \cup \Omega_{\varepsilon/2}$ 
has compact closure 
contained in $\partial\Omega \cup \Omega_\varepsilon$, and both sets
are $C^\infty$.
By \eqref{newnew}, both $u_\lambda$ and the right
hand side $h$ are bounded independently of~$\lambda$.
Hence, by interior and boundary estimates for the linear Poisson equation,
we deduce a bound for $\Vert u_\lambda\Vert_{W^{1,4}(\Omega_{\varepsilon/2})}$ independent
of $\lambda$. Letting $\lambda$ tend to $\lambda^\star$, 
we obtain $u^\star\in L^{\frac{2n}{n-4}}(\Omega)$.

Our second proof is perhaps more direct; it does not use regularity for
the linear problem. Here we choose a regular value $s$ of $u$ 
(and thus $\{x\in\Omega:u_\lambda(x)>s\}$ is smooth) such that
$$
\frac{1}{\gamma} \Vert u^\star\Vert_{L^1 (\Omega)}
\leq s \leq \frac{2}{\gamma} \Vert u^\star\Vert_{L^1 (\Omega)}.
$$ 
By \eqref{newnew} we have
\begin{equation}\label{inclussion}
\Omega_\varepsilon\subset\left\{x\in\Omega: u_\lambda(x)\leq s\right\}.
\end{equation}
Now, we use
$$
\eta(x)=\left\{
\begin{array}{cll}
{\rm dist}(x,\partial\Omega)&\textrm{in}&\Omega_\varepsilon=\{{\rm dist}(x,\partial\Omega)<\varepsilon\},\\
\varepsilon&\textrm{in}&\{{\rm dist}(x,\partial\Omega)\geq\varepsilon\}
\end{array}
\right.
$$
as a test function in \eqref{semi1:new}. Using \eqref{inclussion} we obtain
$$
(n-1)\varepsilon^2\int_{\{u_\lambda>s\}\cap\{|\nabla u_\lambda|>0\}} H_{u_\lambda}^2|\nabla u_\lambda|^2\ dx
\leq \int_{\{u_\lambda<s\}} |\nabla u_\lambda|^2\ dx.
$$
Multiplying equation $\Plambda$ by $T_s u_\lambda=\min\{s,u_\lambda\}$ 
we have
$$
\int_{\{u_\lambda<s\}} |\nabla u_\lambda|^2\ dx=\lambda\int_\Omega f(u_\lambda)T_su_\lambda\ dx
\leq \lambda^\star s \|f(u^\star)\|_{L^1(\Omega)}.
$$
Note that $\|f(u^\star)\|_{L^1(\Omega)}<\infty$ since it is well known that 
$f(u^\star)\ {\rm dist}(\cdot,\partial\Omega)\in L^1(\Omega)$ in 
general smooth domains and if in addition $\Omega$ is convex then
$u^\star$, and thus 
$f(u^\star)$, are bounded in $\Omega_\varepsilon$ by \eqref{newnew}.

Therefore, using Theorem~\ref{Theorem:Sobolev} $(b)$ applied 
to $v=u_\lambda-s$, with $p=2$ and $r=1$, and replacing $\Omega$ by each component of 
$\{x\in\Omega:u_\lambda (x) >s\}$ (which is smooth), we deduce
$$
\left(\int_{\{u_\lambda>s\}}\Big(u_\lambda-s\Big)^{\frac{2n}{n-4}}\ dx\right)^{\frac{n-4}{2n}}
\leq \frac{C_2}{\varepsilon\sqrt{n-1}}\left(\lambda^\star s \|f(u^\star)\|_{L^1(\Omega)}\right)^{\frac{1}{2}}
$$
for all $\lambda\in(0,\lambda^\star)$. In particular, 
letting $\lambda$ tend to $\lambda^\star$, 
we obtain $u^\star\in L^{\frac{2n}{n-4}}(\Omega)$. 

In case that $f$ is only $C^1(\mathbb{R})$ then one can make an easy 
approximation argument to obtain the same result (see proof of Theorem 1.2 in 
\cite{Cabre09} for the details).
\end{proof}

\begin{remark}\label{rem1:7}
As a consequence of Theorem~\ref{Theorem}, if $u\in L^1(\Omega)$ 
is a weak solution of \eqref{problem} (in the sense of Definition \ref{Defn:weak_sol}) 
which is bounded in a neighborhood of $\partial\Omega$ and which is the $L^1(\Omega)$ 
limit of a sequence of classical semi-stable solutions of \eqref{problem}, then 
$u\in L^{2n/(n-4)}(\Omega)$ and $u\in W^{1,p}_0(\Omega)$ for all $p<4n/(3n-4)$. In 
particular, 
$$
u\in L^2(\Omega)\cap W^{1,4/3}_0(\Omega)
$$ 
independently of the dimension $n$. 
For general solutions (not necessarily semi-stable) the best regularity that one 
expects assuming only $g(u)\in L^1(\Omega)$ is $u\in L^q(\Omega)\cap W^{1,p}_0(\Omega)$ 
for all $1\leq q<n/(n-2)$ and $1\leq p<n/(n-1)$. Hence, semi-stable solutions enjoy 
more regularity than general solutions.
\end{remark}

\subsection{Proof of Theorem~\ref{Thm:Nedev}}\label{subsection4:3}
In an unpublished paper, Nedev \cite{Nedev01} proved that the extremal solution 
$u^\star$ lies in the energy class $H^1_0$, independently of the dimension,  
when $\Omega$ is strictly convex. For this, he used a Poho${\rm\check{z}}$aev identity, 
an upper bound independent of $\lambda$ for the energy of the 
minimal solutions $u_\lambda$, and the fact that $u^\star$ is bounded (and hence regular)
in a neighborhood 
of the boundary. Here, for the sake of completeness, we give a proof 
of Nedev's result.

Recall that the energy functional associated to $\Plambda$ is given by 
$$
J_\lambda(u):=\frac{1}{2}\int_\Omega |\nabla u|^2\ dx-\lambda\int_\Omega F(u)\ dx,
\qquad F(u):=\int_0^u f(s)\ ds.
$$
In \cite{Nedev01} an upper bound of $J_\lambda(u_\lambda)$ is proved by using the parabolic 
equation associated to $\Plambda$, $u_t-\Delta u=\lambda f(u)$. This equation was studied 
by Brezis \textit{et al.} \cite{BCMR96}. The proof that we present here uses a different, 
purely elliptic, argument at this point.

\begin{proof}[Proof of Theorem~{\rm\ref{Thm:Nedev}}]
Let $u_\lambda$ be the minimal solution of $\Plambda$ and let $\nu$ be the outward 
unit normal to $\Omega$. 
Multiplying $\Plambda$ by $x\cdot\nabla u_\lambda$ it is standard to obtain 
the following Poho${\rm\check{z}}$aev identity:
\begin{equation}\label{Pohozaev} 
\int_\Omega |\nabla u_\lambda|^2\ dx
=
\frac{1}{2}\int_{\partial\Omega}|\nabla u_\lambda|^2\ x\cdot\nu(x)\ d\sigma+n J_\lambda(u_\lambda)
\end{equation}
for all $\lambda\in(0,\lambda^\star)$. Since the minimal solution $u_\lambda$ is 
the only solution of $\Plambda$ in $\{u\in H^1_0(\Omega):0\leq u\leq u_\lambda\}$, 
it is also the absolute minimizer of $J_\lambda$ in this convex set. Hence, 
we have $J_\lambda(u_\lambda)\leq J_\lambda(0)=0$ for every $\lambda\in(0,\lambda^\star)$. 

Therefore, from \eqref{Pohozaev} one deduces that
\begin{equation}\label{Pohozaev2} 
\int_\Omega |\nabla u_\lambda|^2\ dx
\leq 
\frac{1}{2}\int_{\partial\Omega}|\nabla u_\lambda|^2\ x\cdot\nu(x)\ d\sigma,
\quad\textrm{for all }\lambda\in(0,\lambda^\star).
\end{equation}

Now, since $\Omega$ is convex, there exist positive constants $\varepsilon$ and $\gamma$ depending 
only on the domain $\Omega$ such that \eqref{newnew} holds. As a consequence, 
$\|f(u_\lambda)\|_{L^\infty(\Omega_\varepsilon)}\leq \|f\|_{L^\infty(0,\alpha)}$ for all 
$\lambda\in(0,\lambda^\star)$, 
where $\alpha$ is a constant depending only on $\Omega$ and $\|u^\star\|_{L^1(\Omega)}$ 
---and thus independent of $\lambda$. By \eqref{newnew}, also $u_\lambda$ is bounded 
in $\Omega_\varepsilon$ independently of~$\lambda$.
Hence, using boundary estimates at $\partial\Omega$ for the linear Poisson equation 
$-\Delta u_\lambda=\lambda f(u_\lambda(x))$ in $\Omega_\varepsilon$, we deduce a bound for 
the right hand side of inequality \eqref{Pohozaev2} independent of $\lambda$. Making 
$\lambda$ tend to $\lambda^\star$ we conclude the proof.
\end{proof}

\begin{remark}\label{Rmk:Nedev2}
As mentioned in \cite{Nedev01}, Theorem~\ref{Thm:Nedev} holds for some nonconvex 
domains such as annulus or bean pea shaped domains. Indeed, using 
Poho${\rm\check{z}}$aev identity (obtained multiplying $\Plambda$ by $(x-a)\cdot\nabla u$) 
and the fact that $J_\lambda(u_\lambda)\leq 0$, one obtains 
\begin{equation}\label{Pohozaev3} 
\int_\Omega |\nabla u_\lambda|^2\ dx
\leq 
\frac{1}{2}\int_{\partial\Omega}|\nabla u_\lambda|^2(x-a)\cdot\nu(x)\ d\sigma
\quad\textrm{for all }\lambda\in(0,\lambda^\star).
\end{equation}
Let $E:=\{x\in\partial\Omega:$ there exists $\varepsilon>0$ and 
a hyperplane $P$ such that $P\cap\overline{\Omega}\cap B_\varepsilon(x)
=\{x\}\}$. By using the moving planes method, as in the proof of Proposition
\ref{Prop4}, it can be seen that $u^\star$ is bounded (by a constant independent 
of $\lambda$) and regular in a neighborhood in $\Omega$ of any compact subset 
of $E$. In particular, if there exists $a\in \mathbb{R}^n$ and $\alpha<0$ such 
that $(x-a)\cdot \nu(x)\leq \alpha$ for every $x\in \partial\Omega\setminus E$ 
one obtains from \eqref{Pohozaev3} that $u^\star\in H^1_0(\Omega)$.
\end{remark}

\bibliographystyle{amsplain}

\end{document}